\documentclass[11pt,a4paper]{article}
\usepackage{amsfonts,amsgen,amsmath,amstext,amsbsy,amsopn,amsthm,amsfonts,amssymb,amscd}
\usepackage{epsf,epsfig}
\usepackage{float}
\usepackage{ebezier,eepic}
\usepackage{color}
\usepackage{tikz}
\usepackage{multirow}
\usepackage{graphicx}
\usepackage{subfigure}
\usepackage{mathrsfs}
\usepackage{graphics}
\usepackage{graphicx}
\usepackage{color}
\usepackage{ifpdf}
\usepackage{fancyhdr}
\usepackage{epstopdf}
\usepackage{epsfig}
\usepackage{indentfirst}
\usepackage{CJK}
\usepackage{chngcntr}

\newtheorem{thm}{Theorem}[section]

\newtheorem{prop}[thm]{Proposition}
\newtheorem{lem}[thm]{Lemma}
\newtheorem{false statement}{False statement}
\newtheorem{cor}[thm]{Corollary}
\newtheorem{fact}[thm]{Fact}

\theoremstyle{definition}

\newtheorem{claim}{Claim}

\newtheorem{conj}[thm]{Conjecture}

\counterwithin{equation}{section}

\newcommand{\ex}{{\rm ex}}

\def\ff{\mathcal{F}}
\def\hh{\mathcal{H}}

\def\hs{\mathcal{S}}

\begin{document}
\title{Hypergraph Tur\'an problem of the generalized triangle with bounded matching number}
\author{
Jian Wang\footnote{Department of Mathematics, Taiyuan University of Technology, Taiyuan 030024, P. R. China. E-mail:wangjian01@tyut.edu.cn. Research supported by NSFC No.12471316 and Natural Science Foundation of Shanxi Province, China No. RD2200004810.} \quad\quad
Wenbin Wang\footnote{Department of Mathematics, Taiyuan University of Technology, Taiyuan 030024, P. R. China. E-mail:wangwb1017@126.com. }\quad\quad
Weihua Yang\footnote{Department of Mathematics, Taiyuan University of Technology, Taiyuan 030024, P. R. China. E-mail:yangweihua@tyut.edu.cn. Research supported by NSFC No.11671296. }\\
 }
\date{}
\maketitle

\begin{abstract}

Let $\mathcal{H}$ be a 3-graph on $n$ vertices. The matching number $\nu(\mathcal{H})$ is defined as the maximum number of disjoint edges in $\mathcal{H}$. The generalized triangle $F_5$ is a 3-graph on the vertex set $\{a,b,c,d,e\}$ with the edge set $\{abc, abd,cde\}$. In this paper, we showed that an $F_5$-free 3-graph $\mathcal{H}$ with matching number at most $s$ has at most $s\lfloor (n-s)^2/4\rfloor$ edges for $n\geq 30(s+1)$ and $s\geq 3$. For the proof, we establish a  2-colored version of Mantel's theorem, which may be of independent interests.

\end{abstract}


\section{Introduction}
For $r\geq 2$, an $r$-graph $\mathcal{H}=(V(\mathcal{H}),E(\mathcal{H}))$ consists of a vertex set $V(\mathcal{H})$ and an edge set $E(\mathcal{H}),$ where $E(\mathcal{H})$ is a collection of $r$-element subsets of $V(\mathcal{H})$. For $r\geq 3$, we often identify $\mathcal{H}$ with $E(\mathcal{H})$.   Given a family $\mathcal{F}$  of $r$-uniform hypergraphs, an $r$-graph $\mathcal{H}$ is called {\it  $\mathcal{F}$-free} if it contains no member of $\mathcal{F}$ as a subgraph. The Tur\'{a}n number $\ex_r(n,\ff)$ is defined as the maximum number of edges in an $\ff$-free $r$-graph on $n$ vertices.  Let $T_r(n,\ell)$ denote the {\it  generalized Tur\'{a}n graph}, which is the complete $\ell$-partite $r$-graph on $n$ vertices with $\ell$ partition classes, with each of size either $\lfloor \frac{n}{\ell}\rfloor$ or
$\lceil \frac{n}{\ell}\rceil$. Denote by $t_r(n,\ell)$ the number of edges in $T_r(n,\ell)$. The classical Tur\'{a}n theorem~\cite{turan1941}  states that for $\ell\geq 2$, $\ex_2(n,K_{\ell+1})$ is uniquely achieved by $T_2(n,\ell)$.

 For an $r$-graph $\mathcal{H}$, the {\it matching number} $\nu(\mathcal{H})$ is defined as the maximum number of pairwise disjoint edges in $\mathcal{H}$. Let $M_{s+1}^{r}$ denote a matching of size $s+1$ with each edge being an $r$-set.
 In 1959, Erd\H{o}s-Gallai \cite{EG} determined  the  Tur\'{a}n number of $M_{s+1}^{2}$.

\begin{thm}[\cite{EG}]\label{thm-eg}
For $n\geq 2s+1$, $${\rm ex}_2(n,M_{s+1}^2)= \max\left\{\binom{2s+1}{2},\binom{s}{2}+s(n-s)\right\}.$$
\end{thm}

Using the shifting technique, Akiyama and Frankl \cite{AF} give a short proof of Theorem \ref{thm-eg} and  almost the same proof yields  a rainbow extension of Theorem \ref{thm-eg}.

\begin{lem}[\cite{AF}]\label{Lemma-rainbow matching}
  Let  $G_{1},\ldots,G_{s}\subset\binom{[n]}{2}$ with $n> 2s$. If  $e(G_i)>\max\{\binom{2s-1}{2},\binom{s-1}{2}+(s-1)(n-s+1)\}$ for all $i=1,2,\ldots,s$, then there exist pairwise disjoint edges $e_1\in E(G_1),\ldots,e_s\in E(G_s)$.
\end{lem}

Let $G(n,\ell,s)$ denote the complete $\ell$-partite graph on $n$ vertices, with one of the partite sets  being of size $n-s$ and the other $\ell-1$ of them inducing a copy of $T_2(s,\ell-1)$. Let $g(n,\ell,s)$ denote the number of edges in $G(n,\ell,s)$.

In 2024, Alon and Frankl \cite{AlonFrankl} determined  the maximum number of  edges in a graph with bounded clique number and bounded matching number.

\begin{thm}[\cite{AlonFrankl}]\label{thm-AF24}
	For $n\geq 2s+1, \ell\geq 2.$
	\[
	\ex_2(n,\{K_{\ell+1},M_{s+1}^2\}) = \max\{t_2(2s+1,\ell),g(n,\ell,s)\}.
	\]
\end{thm}

Let us mention that the case $n>3s$ of Theorem \ref{thm-AF24} was also obtained in \cite{FWY}.

A {\it weakly independent set} of an $r$-graph $\hh$ is a subset of $V(\hh)$ that does not contain any edge of $\hh$.  A proper $k$-coloring of $\hh$ is a mapping from $V(\hh)$ to a set of $k$ colors such that no edge of $\hh$ is monochromatic. The {\it chromatic number} $\chi(\hh)$ is defined as  the minimum number of colors needed for proper coloring $\hh$.

Recently, Gerbner, Tompkins and  Zhou \cite{GTZ} considered the analogous Tur\'{a}n problems for hypergraphs with bounded matching number.

\begin{thm} [\cite{GTZ}]
If $\chi(F)>2$ and $n$ is sufficiently large, then
\[
\ex_r(n,\{F,M_{s+1}^{r}\}) =\ex_r(s,\mathcal{F}) +\sum_{1\leq i\leq t}\binom{s}{i} \binom{n-s}{r-i},
\]
where $t=\min\{s,r-1\}$ and $\mathcal{F}$ is the family of $r$-graphs obtained by deleting a weakly independent set from $F$.
\end{thm}

The generalized triangle $F_5$ is a 3-graph on the vertex set $\{a,b,c,d,e\}$ with the edge set $\{abc, abd,cde\}$. It is easy to check that $\chi(F_5)=2$. However, unlike the graph case, $\chi(F_5)=2$ does not mean that $F_5$ is degenerate. As proved by Frankl and F\"{u}redi \cite{FranklFuredi}, $F_5$ has positive Tur\'{a}n density.

\begin{thm}[\cite{FranklFuredi}]\label{thm-FF}
	For $n\geq 3000$,
	\begin{align}\label{ineq-turanF5}
	\ex_3(n,F_5) = \left\lfloor \frac{n}{3}\right\rfloor \left\lfloor \frac{n+1}{3}\right\rfloor \left\lfloor \frac{n+2}{3}\right\rfloor.
	\end{align}
\end{thm}
Later,  the condition $n\geq 3000$ was improved to $n\geq 33$ by Keevash and Mubayi \cite{PeterMubayi}.

 Let $\mathcal{H}_{3}(n,s)$ denote the complete $3$-partite $3$-graph on $n$ vertices with one partite set of size $s$ and the other two partite sets of sizes  $\lfloor\frac{n-s}{2}\rfloor$ and $\lceil\frac{n-s}{2}\rceil$.  Clearly  $|\mathcal{H}_{3}(n,s)|=s \lfloor\frac{(n-s)^{2}}{4}\rfloor$. It is easy to check that $\mathcal{H}_{3}(n,s)$ is $F_5$-free and has matching number at most $s$. For $x\in [n]$, define the {\it full star of center $x$} as the 3-graph $\hs_x=\{E\in \binom{[n]}{3}\colon x\in E\}$. Clearly, $\hs_x$ is $\{F_{5},M_{2}^{3}\}$-free.

 In this paper, we show that $\mathcal{H}_{3}(n,s)$ attains the maximum number of edges among all $n$-vertex $F_5$-free 3-graphs with matching number $s$ for $n\geq 30(s+1)$ and $s\geq 3$.

\begin{thm}\label{main1}
	For $n\geq 30(s+1)$,
	\[
	\ex_3(n,\{F_{5},M_{s+1}^{3}\}) = \left\{
                \begin{array}{ll}
                   \binom{n-1}{2}, & \mbox{ if } s=1,2; \\[5pt]
                   s \left\lfloor\frac{(n-s)^{2}}{4}\right\rfloor, & \mbox{ if } s\geq 3.
                \end{array}
              \right.
	\]
Moreover, for $s=1,2$ the full star is  the unique $\{F_{5},M_{s+1}^{3}\}$-free 3-graph attaining the maximum size and for $s\geq 3$, $\mathcal{H}_{3}(n,s)$ is the unique $\{F_{5},M_{s+1}^{3}\}$-free 3-graph attaining the maximum size up to isomorphism.
\end{thm}

Let $G_1,G_2,\ldots,G_s$ be simple graphs on the same vertex set $V$. Define an {\it edge-colored multiple graph}  $G_1+G_2+\cdots+G_s$  as a multiple graph on the vertex set $V$ with edges in $E(G_i)$ having color $i$ for $i=1,2,\ldots,s$.  We call an edge-colored graph \textit{rainbow} if no two edges have the same color. We say that $G$ is {\it rainbow $H$-free} if  it does not contain a rainbow $H$ in $G$. For $G=G_1+G_2+\cdots+G_s$, we also say that $G_1,G_2,\ldots,G_s$ are rainbow $H$-free.

Frankl \cite{rainbow triangle} proved a rainbow version of Mantel's Theorem.

\begin{thm}[\cite{rainbow triangle}]\label{rainbow triangle}
	Let $p\geq 3$ and let $G_{1},\ldots,G_{p}\subset\binom{[n]}{2}$ be  rainbow triangle free. Then
	\[
	\sum_{i=1}^{p}e(G_i)\leq \max \left\{\,2\binom{n}{2} ,\, p\left\lfloor\frac{n^2}{4}\right\rfloor\right\}.
	\]
\end{thm}

Let us mention that the corresponding extremal problem for rainbow $H$-free graphs was studies systematically in \cite{KSSV}. See also \cite{MH}, \cite{LMZ} et al. for more recent results.

We call  an edge-colored triangle {\it 2-colored} if its two edges have the same color and the third edge has a different color.
As a key ingredient to the proof of Theorem \ref{main1}, we proved a 2-colored version of  Mantel's Theorem.

\begin{thm}\label{thm-2-colored}
	Let $p\geq 2$ and let $G_{1},\ldots,G_{p}\subset\binom{[n]}{2}$ be 2-colored triangle free. Then
	\[
	\sum_{i=1}^{p}e(G_i)\leq p \cdot \left\lfloor\frac{n^{2}}{4}\right\rfloor
	\]
Moreover, the equality holds if and only if $G_{1}=\cdots=G_{p}=K_{\lfloor \frac{n}{2}\rfloor, \lceil \frac{n}{2}\rceil}$.
\end{thm}

Note that  both  Theorem \ref{rainbow triangle} and Theorem \ref{thm-2-colored} imply Mantel's Theorem by setting $G_1=G_2=\ldots=G_p$.

\begin{cor}
	Let $p\geq 2$ and let $G_{1},\ldots,G_{p}\subset\binom{[n]}{2}$ be 2-colored triangle free. Then
	\[
	\prod_{i=1}^{p}e(G_i)\leq  \left\lfloor\frac{n^{2}}{4}\right\rfloor^p.
	\]
Moreover, the equality holds if and only if $G_{1}=\cdots=G_{p}=K_{\lfloor \frac{n}{2}\rfloor, \lceil \frac{n}{2}\rceil}$.
\end{cor}
\begin{proof}
For any $1\leq i<j\leq p$, by Theorem \ref{thm-2-colored} we have
\[
e(G_i)e(G_j)\leq \left(\frac{e(G_i)+e(G_j)}{2}\right)^2\leq  \left\lfloor\frac{n^{2}}{4}\right\rfloor^2.
\]
Thus the corollary follows.
\end{proof}

{\noindent\bf Remark.} Let us mention that in \cite{rainbow triangle} Frankl conjectured that if  $G_{1},G_2,G_3\subset\binom{[n]}{2}$ is  rainbow triangle free, then
\[
e(G_1)e(G_2)e(G_3)\leq  \left\lfloor\frac{n^{2}}{4}\right\rfloor^3.
\]
However, it was  disproved by Frankl, Gy\H{o}ri, He, Lv, Salia, Tompkins, Varga and
Zhu \cite{rainbowtriangle2}.

Let us recall some definitions and notations. Let   $\mathcal{H}$ be a 3-graph and let $u\in V (\mathcal{H})$.  The {\it link graph} $L_\mathcal{H}(u)$ is defined as
\[
L_\mathcal{H}(u)=\left\{vw\in \binom{V (\mathcal{H})}{2}\colon uvw\in \mathcal{H}\right\}.
\]
The {\it degree } $\deg_{\mathcal{H}}(u)$ of $u$ is defined as the number of edges in $L_{\mathcal{H}}(u)$. For $W\subset V (\mathcal{H})$, we also use
\[
L_\mathcal{H}(u,W)= \left\{vw\in L_{\mathcal{H}}(u)\colon \{v,w\}\subset W\right\}.
\]
and $\deg_{\hh}(u,W)=|L_\hh(u,W)|.$ Let $N_{\hh}(v)$ denote the set of vertices $u$ such that $uv$ is covered by an edge of $\hh$.    For $\{u,v\}\subset V (\mathcal{H}),$ the {\it neighborhood } $N_\hh(u,v)$ is defined as
\[
N_\hh(u,v)=\left\{w\in V(\hh)\colon uvw\in \mathcal{H}\right\}.
\]
and $\deg_{\mathcal{H}}(u,v)=|N_\hh(u,v)|$. When the context is clear, we often omit the subscript $\hh$.

For  $S\subseteq V(\hh),$ let $\hh[S]$ denote the subgraph of $\hh$ induced by $S$ and let $\hh-S$ denote the subgraph $\hh[V(\hh)\setminus S]$. For $S,S'\subseteq V(\hh)$ and $S\cap S'=\emptyset$,
let $\hh[S, S']$ denote the subgraph  on the vertex set  $S\cup S'$ with the edge set consisting of all  $e\in E(\hh)$ intersecting both $S$ and $S'$.

\section{Proof of Theorem \ref{thm-2-colored}}

\begin{fact}
Let $p\geq 2$ and let $G_{1},\ldots,G_{p}\subset\binom{[n]}{2}$ be 2-colored triangle free.  If $xy\in E(G_1)$ and $N_{G_1}(x)\cap N_{G_1}(y)\neq \emptyset$, then for $2\leq i\leq p$,
\begin{align}\label{ineq-key1}
\deg_{G_i}(x) + \deg_{G_i}(y)+\deg_{G_1}(x) + \deg_{G_1}(y) \leq 2n-2.
\end{align}
\end{fact}
\begin{proof}
Let  $z\in N_{G_1}(x)\cap N_{G_1}(y)$. Then $xyz$ forms a triangle in $G_1$. Since $G_{1},\ldots,G_{p}$ are 2-colored triangle free, $xy\notin E(G_i)$. If $w\in N_{G_1}(x)\cap N_{G_i}(y)$, then $xy,xw\in E(G_1)$ and $yw\in E(G_i)$. That is, $xyw$ forms a 2-colored triangle, a contradiction. Thus $N_{G_1}(x)\cap N_{G_i}(y)=\emptyset$. Similarly, we have $N_{G_i}(x)\cap N_{G_1}(y)=\emptyset$. Note that $x\notin N_{G_1}(x)\cup N_{G_i}(y)$ and $y\notin N_{G_i}(x)\cup N_{G_1}(y)$. It follows that
\[
\deg_{G_1}(x) +\deg_{G_i}(y) \leq n-1 \mbox{ and } \deg_{G_i}(x) +\deg_{G_1}(y) \leq n-1.
\]
Thus \eqref{ineq-key1} follows.
\end{proof}

The following inequality is standard and useful.

\begin{align}\label{ineq-degreepower}
\sum_{v\in V(G)} \deg_{G}(v)^2\geq  \frac{4e^2(G)}{|V(G)|}.
\end{align}

\begin{proof}
Since $x^2$ is a convex function, by the Jensen's inequality
\[
\frac{1}{|V(G)|}\sum_{v\in V(G)} \deg_{G}(v)^2 \geq \left(\frac{\sum_{v\in V(G)} \deg_{G}(v)}{|V(G)|}\right)^2 = \frac{4e^2(G)}{|V(G)|^2}
\]
and \eqref{ineq-degreepower} follows.
\end{proof}

\begin{proof}[Proof of Theorem \ref{thm-2-colored}]
	We prove the theorem by induction on $n$. For $n=1$ and $n=2,$ the theorem holds trivially. For $n=3$, it is easy to check that $\sum_{i=1}^{p}e(G_i)\leq 2p= p\lfloor\frac{n^{2}}{4}\rfloor$. Now assume that the theorem holds for $1,2,\ldots, n-1$ and we prove it for $n$. We distinguish two cases.

{\bf Case 1. } $\deg_{G_i}(x)+\deg_{G_i}(y)\leq n$ for all $xy\in E(G_i)$ and $i\in [p]$.

Note that by \eqref{ineq-degreepower},
	\begin{equation}\label{2deg=e}
\begin{aligned}
		\sum_{xy\in E(G_i)}(\deg_{G_i}(x)+\deg_{G_i}(y))
		&=\sum_{x\in V(G_i)} \deg_{G_i}^{2}(x)\geq  \frac{4e^{2}(G_i)}{n}.
	\end{aligned}
	\end{equation}
It follows that
	\[
	\frac{4e^{2}(G_i)}{n} \leq \sum_{xy\in E(G_i)}(\deg_{G_i}(x)+\deg_{G_i}(y))\leq n\cdot e(G_i).
	\]
Thus $e(G_i)\leq \lfloor\frac{n^{2}}{4}\rfloor$ for each $i\in [p]$ and
	$\sum_{i=1}^{p}e(G_i)\leq p\cdot e(G_i)\leq p\cdot\lfloor\frac{n^{2}}{4}\rfloor$ follows. Moreover, the equality holds if and only if $\deg_{G_i}(x)+\deg_{G_i}(y)= n$ holds for all $xy\in E(G_i)$ and $i\in [p]$. If there is a triangle $xyz\in E(G_i)$ then by \eqref{ineq-key1} we have
\[
\deg_{G_i}(x) + \deg_{G_i}(y)+\deg_{G_j}(x) + \deg_{G_j}(y) \leq 2n-2 \mbox{ for } j\neq i,
\]
a contradiction. Thus the equality holds  only if each $G_i$ is triangle free. By Mantel's theorem, each $G_i$ is a $K_{\lfloor \frac{n}{2}\rfloor, \lceil \frac{n}{2}\rceil}$. Therefore the equality holds if and only if $G_{1}=\cdots=G_{p}=K_{\lfloor \frac{n}{2}\rfloor, \lceil \frac{n}{2}\rceil}$.

{\bf Case 2. } There exist $i\in [p]$ and $xy\in E(G_i)$ such that $\deg_{G_i}(x)+\deg_{G_i}(y)> n$.

Choose $i$ and $uv\in E(G_i)$ such that  $\deg_{G_i}(u)+\deg_{G_i}(v)$ is maximal. Assume that $\deg_{G_i}(u)+\deg_{G_i}(v)=n+t$. Then
for any $j\neq i$, by \eqref{ineq-key1}
\[
\deg_{G_j}(u)+\deg_{G_j}(v) \leq 2n-2 -n-t =n-t-2.
\]
It follows that for $p\geq 2$
\[
\sum_{1\leq \ell \leq p} \left(\deg_{G_\ell}(u)+\deg_{G_\ell}(v) \right) \leq n+t+(p-1)(n-t-2) \leq p(n-1).
\]
Then by the induction hypothesis,
	\begin{align*}
		\sum_{1\leq \ell \leq p}e(G_\ell)&\leq \sum_{1\leq \ell \leq p}e(G_\ell-\{u,v\})+ \sum_{ \ell\neq i}\left(\deg_{G_\ell}(u)+\deg_{G_\ell}(v)\right)+ \left(\deg_{G_i}(u)+\deg_{G_i}(v)-1\right)\\[3pt]
		&\leq p\cdot\left\lfloor\frac{(n-2)^{2}}{4}\right\rfloor+ p\cdot(n-1) -1 \\[3pt]
		&= p\cdot\left\lfloor\frac{n^{2}}{4}\right\rfloor-1
	\end{align*}
\end{proof}

We have also need  the follow result.
\begin{thm}\label{thm-disjoint}
Let $0<\alpha, \beta<\frac{1}{4}$ be reals satisfying $\beta+2\alpha^{\frac{3}{4}}<\frac{1}{4}$ and let $p\geq 2$, $n>\frac{4}{1-4\beta}$. Suppose that $G_{1},\ldots,G_{p}\subset\binom{[n]}{2}$ are $2$-colored triangle free. If $E(G_i)\cap E(G_j)=\emptyset$ and $e(G_i),e(G_j)>\frac{n^2}{4}-\alpha n^2$  for some $1\leq i<j\leq p$, then
\begin{align}\label{ineq-2.1}
\sum_{i=1}^{p}e(G_i)\leq p \cdot \frac{n^{2}}{4}-(p-2)\beta n^2.
\end{align}
\end{thm}
\begin{proof}
By symmetry assume  $E(G_1)\cap E(G_2)=\emptyset$ and $e(G_1),e(G_2)>\frac{n^2}{4}-\alpha n^2.$ Then obviously $e(G_1)+e(G_2)\leq \binom{n}{2}< \frac{n^{2}}{2}. $ If $e(G_k)\leq \frac{n^2}{4}-\beta n^2$ for all $k\in\{3,4,\ldots,p\}$, then \eqref{ineq-2.1} follows.
Thus by symmetry we may assume $e(G_3)>\frac{n^2}{4}-\beta n^2$.

\begin{claim}\label{claim-prei-2}
	$\Delta(G_3)<2\sqrt{\alpha}n$.
\end{claim}
	\begin{proof}
Suppose that there exists $v\in [n]$ with $\deg_{G_3}(v)\geq 2\sqrt{\alpha}n$. Let $X=N_{G_3}(v)$. Since $G_{1},\ldots,G_{p}$ are $2$-colored triangle free, $X$ is an independent set of $G_k$ for all $k\neq  3$.  Thus,
\begin{align*}
e(G_1)+e(G_2)\leq \binom{n}{2}-\binom{|X|}{2}\leq \binom{n}{2}-\binom{2\sqrt{\alpha}n}{2}<\frac{n^2}{2}-2\alpha n^2
\end{align*}
contradicting the fact that $e(G_i)>\frac{n^2}{4}-\alpha n^2$, $i=1,2$.
\end{proof}

Define $G_3^2$ as the square graph of $G_3$ such that  $xy\in E(G_3^2)$ if and only if there exists $z\in [n]$ such that $xz,yz\in E(G_3)$.  Since $\Delta(G_3)<2\sqrt{\alpha}n$, for any $xy\in E(G_3^2)$ there are at most $2\sqrt{\alpha}n$ choices of $z$ with $xz,yz\in E(G_3)$. Thus,
\begin{equation}\label{vertexpair}
\begin{aligned}
		 e(G_3^2) \geq \frac{1}{2\sqrt{\alpha}n}\cdot\sum_{i=1}^{n}\binom{\deg_{G_3}(i)}{2}
&= \frac{1}{4\sqrt{\alpha}n}\cdot\sum_{i=1}^{n}\left(\deg_{G_3}^2(i)-\deg_{G_3}(i)\right) \\[3pt]
&\overset{\eqref{ineq-degreepower}}{>}\frac{1}{4\sqrt{\alpha}n}\cdot\left(\frac{4e^{2}(G_3)}{n}-2e(G_3)\right)\\[3pt]
        &=\frac{e(G_3)}{\sqrt{\alpha}n^2}\left(e(G_3)-\frac{n}{2}\right).
\end{aligned}
\end{equation}
Note that $n>\frac{4}{1-4\beta}$ implies $e(G_3)>\frac{n^2}{4}-\beta n^2>n$. It follows that $e(G_3)-\frac{n}{2}>\frac{e(G_3)}{2}$.  Then by $e(G_3)> \frac{n^2}{4}-\beta n^2>2\alpha^{\frac{3}{4}}n^2$,
\[
 e(G_3^2) \geq \frac{e(G_3)^2}{2\sqrt{\alpha}n^2}>\frac{(2\alpha^{\frac{3}{4}}n^2)^2}{2\sqrt{\alpha}n^2} =2\alpha n^2.
\]
Since $G_1,G_2,G_3$ are 2-colored triangle free, we infer that $(E(G_1)\cup E(G_2))\cap E(G_3^2)=\emptyset$. Thus,
\[
e(G_1)+e(G_2)\leq \binom{n}{2}-e(G_3^2)\leq \binom{n}{2}- 2\alpha n^2.
\]
contradicting $e(G_i)>\frac{n^2}{4}-\alpha n^2$, $i=1,2$. Therefore, $e(G_k)\leq \frac{n^2}{4}-\beta n^2$ for all $k=3,4,\ldots,p$ and $\sum_{i=1}^{p}e(G_i)\leq p \cdot \frac{n^{2}}{4}-(p-2)\beta n^2$.
\end{proof}

\section{The case $s\leq 2$}

We  need the following two classical results from extremal set theory.
\begin{thm}[The Exact Erd\H{o}s-Ko-Rado Theorem \cite{EKR},\cite{F-EKR},\cite{W-EKR}]\label{thm-EKR}
  Let $\ff\subset \binom{[n]}{k}$ be a $t$-intersecting family. If $n>(t+1)(k-t+1)$ then
  \[
  |\ff|\leq \binom{n-t}{k-t}.
  \]
\end{thm}

\begin{thm}[The Hilton-Milner Theorem \cite{HM}]\label{thm-HM}
  Let $n\geq 2k.$ If $\ff\subset \binom{[n]}{k}$ is an intersecting family that is not a star, then
  \[
  |\ff|\leq \binom{n-1}{k-1}-\binom{n-k-1}{k-1}+1.
  \]
\end{thm}

For any $X\subset [n]$, we use $\bar{X}$ to denote $[n]\setminus X$.  The following lemma plays a crucial role in our proof.

\begin{lem}\label{lem-e(H)}
Let $\hh\subset \binom{[n]}{3}$ be a 3-graph with $\nu(\hh)\leq s$, $s\geq 2$ and $n\geq 3(s+1)$. Then there exist $0\leq i \leq s$ and $A\subset S\subset [n]$ with $|A|=s-i$ and $|S|=s$ such that $\deg(x)\leq (s+2i)n$ for all $x\in \bar{A}$ and
\[
|\hh|\leq \sum_{v\in A} \deg(v,\bar{S})+\sum_{\{u,v\}\in \binom{A}{2}}\deg(u,v)+isn+6i^2n.
\]
\end{lem}
\begin{proof}
Without loss of generality, assume that $\deg(1)\geq\deg(2)\geq\ldots \geq\deg(n)$.

\begin{claim}\label{claim-3.1}
There exists $0\leq i\leq s$ such that $\deg(s+1-i)\leq (s+2i)n$.
\end{claim}

\begin{proof}
Suppose for contradiction that $\deg(s+1-i)> (s+2i)n$ holds for all $i=0,1,2,\ldots,s$. Then one can find a matching of size $s+1$ by the following procedure. In the first step, by $\deg(s+1)> sn$ we can choose $e_{s+1}$ with $s+1\in e_{s+1}$ such that $e_{s+1} \cap [s+1]=\{s+1\}$. In the second step by $\deg(s)> (s+2)n$ there is  $e_{s}$ with $s\in e_{s}$ such that $e_{s} \cap ([s+1]\cup e_{s+1})=\{s\}$.  In the $(j+1)$th step with $j=0,1,2,\ldots,s$, by $\deg(s+1-j)> (s+2j)n$ there is  $e_{s+1-j}$ with $s+1-j\in e_{s+1-j}$ such that $e_{s} \cap ([s+1]\cup e_{s+1}\cup \ldots\cup e_{s+2-j})=\{s+1-j\}$. Eventually we will obtain a matching $e_{s+1},e_s,\ldots,e_1$, contradicting $\nu(\hh)\leq s$.
\end{proof}

By Claim \ref{claim-3.1}, we may choose the maximal $i$ such that $\deg(s+1-i)\leq (s+2i)n$. Let $A=[s-i]$ and $\bar{A}=[n]\setminus A$.
Let $\nu(\hh[\bar{A}])=k$ and let $e_1,e_2,\ldots,e_k$ be a maximal matching in $\hh[\bar{A}]$.

\begin{claim}\label{claim0}
$\nu(\hh[\bar{A}])=k\leq i$.
\end{claim}

\begin{proof}
If $k \geq i+1$, then by $\deg(j) > (3s+2-2j)n$ for $j \in [s-i]$, one can extend $e_1,e_2,\ldots,e_k$ to a matching of size $s+1$ in $\mathcal{H}$, contradicting $\nu(\mathcal{H}) \leq s$.
\end{proof}

 Let $e_1\cup e_2\cup \ldots\cup e_k =\{u_1,u_2,\ldots,u_{3k}\}=X$ and assume
 \[
 \deg_{\hh[\bar{A}]}(u_1)\geq\deg_{\hh[\bar{A}]}(u_2)\geq\cdots \geq\deg_{\hh[\bar{A}]}(u_{3k}).
 \]
 Since $e_1,e_2,\ldots,e_k$ form a maximal matching in $\hh[\bar{A}]$, we infer that $e\cap X\neq \emptyset$ for all $e\in \hh[\bar{A}]$.

If $\deg_{\hh[\bar{A}]}(u_{i+1})>2in$, then let $W=\bar{A}\setminus\{u_1,\ldots,u_{i+1}\}$ and $\deg_{\hh[W]}(u_{j})>in$ holds for each $j=1,2,\ldots,i+1$.  By Lemma \ref{Lemma-rainbow matching}, there is an $(i+1)$-matching in $\hh[\bar{A}]$, contradicting Claim \ref{claim0}. Thus,
 \[
  \deg_{\hh[\bar{A}]}(u_{3k})\leq \cdots\leq   \deg_{\hh[\bar{A}]}(u_{i+2})\leq \deg_{\hh[\bar{A}]}(u_{i+1})\leq 2 i n.
 \]

Now let $B=\{u_1,u_2,\ldots,u_i\}$. Recall that $\deg(u)\leq (s+2i)n$ for  $u\in B$. Thus,
\[
|\hh| \leq \sum_{u\in B} \deg(u) +|\hh[\bar{B}]| \leq i(s+2i)n+|\hh[\bar{B}]|.
\]
Let $S=A\cup B$.  Then
\begin{align*}
|\hh[\bar{B}]| &\leq \sum_{v\in A} \deg(v,\bar{S}) +\sum_{\{u,v\}\in \binom{A}{2}}\deg(u,v) +|\hh[\bar{S}]| \\[3pt]
&\leq  \sum_{v\in A} \deg(v,\bar{S}) +\sum_{\{u,v\}\in \binom{A}{2}}\deg(u,v) + \sum_{i+1\leq j\leq 3k} \deg_{\hh[\bar{A}]}(u_j)\\[3pt]
&\leq  \sum_{v\in A} \deg(v,\bar{S}) +\sum_{\{u,v\}\in \binom{A}{2}}\deg(u,v) + 2 i \cdot 2 i n.
\end{align*}
Thus,
\[
|\hh| \leq \sum_{v\in A} \deg(v,\bar{S}) +\sum_{\{u,v\}\in \binom{A}{2}}\deg(u,v) + i s n+ 6i^2n.
\]
\end{proof}

We now prove the $s\leq 2$ case of Theorem \ref{main1}.

\begin{prop}\label{prop-1}
For $n\geq 60$ and $s=1,2$,
\[
\ex_3(n,\{F_{5},M_{s+1}^{3}\}) = \binom{n-1}{2}
\]
and the full star is the unique $\{F_{5},M_{s+1}^{3}\}$-free 3-graph attaining the maximum size.
\end{prop}

\begin{proof}
Note that the $s=1$ case of the proposition  follows  from Theorem \ref{thm-EKR}. Thus we are left with  the $s= 2$ case.
Let $\mathcal{H}\subset \binom{[n]}{3}$ be an $F_5$-free $3$-graph with  $\nu(\hh)\leq 2$. If $\nu(\hh)=1$, then by Theorem \ref{thm-EKR}
\[
|\hh| \leq \binom{n-1}{2}
\]
and we are done. Thus we may assume  $\nu(\hh)= 2$.

By Lemma \ref{lem-e(H)}, there exist $0\leq i \leq 2$ and $A\subset S\subset [n]$ with $|A|=2-i$ and $|S|=2$ such that $\deg(x)\leq 2(i+1)n$ for all $x\in \bar{A}$ and
\[
|\hh|\leq \sum_{v\in A} \deg(v,\bar{S})+\sum_{\{u,v\}\in \binom{A}{2}}\deg(u,v)+2 i n+6i^2n.
\]
If $i=2$, then $A=\emptyset$ and for $n\geq 60$,
 \[
 |\hh|\leq 28n< \binom{n-1}{2}.
 \]
Thus we may assume $i=1$ or $i=0$. We distinguish two cases.

  {\bf Case 1. } $i=1$.

  Assume  $A=\{1\}$ and let $\tilde{\hh}=\hh-\{1\}$, $G=L(1)$. Then by $\nu(\hh)= 2$ we infer that $\tilde{\hh}$ is intersecting and non-empty.

 \begin{claim}\label{claim s=2}
     For each $x_1x_2x_3\in \tilde{\hh}$  and  $v\in [2,n]\setminus  \{x_1, x_2, x_3\}$, there is at most one edge in $G$ between $v$ and  $\{x_1,x_2,x_3\}$.
  \end{claim}
  \begin{proof}
 If $vx_1,vx_2 \in E(G)$ then $x_1x_2x_3,1vx_1,1vx_2$ form an $F_5$, a contradiction.
  \end{proof}

  Recall that $\nu(\hh)= 2$ implies $\tilde{\hh}\neq \emptyset$. By Claim \ref{claim s=2} we have
  \[
  e(G) \leq \binom{n-1}{2} - 2(n-4).
  \]
  If $|\tilde{\hh}|< 2(n-4)$, then
  \[
  |\hh| = e(G)+|\tilde{\hh}| <  \binom{n-1}{2}.
  \]
  Thus we may assume $|\tilde{\hh}|\geq 2(n-4)$.

  If $\tilde{\hh}$ is 2-intersecting, then by Theorem \ref{thm-EKR}
    \begin{align*}
     |\tilde{\hh}|\leq \binom{(n-1)-2}{3-2}=n-3<2(n-4),
    \end{align*}
    a contradiction. Thus there exist $x_1x_2x_3,x_3x_4x_5\in \tilde{\hh}$. Then by Claim \ref{claim s=2}, for each $x\in [2,n]\setminus \{x_1,x_2,x_3,x_4,x_5\}$ there are at most 2 edges in $G$ between $x$ and $\{x_1,x_2,x_3,x_4,x_5\}$. Moreover, there are at most 2 edges in $G$ between $\{x_1,x_2,x_3\}$ and $\{x_4,x_5\}$. Thus,
    \begin{align}\label{ineq-3.1}
    e(G) \leq \binom{n-1}{2} - 3(n-6)- 4 = \binom{n-1}{2}-3n +14.
    \end{align}

    {\bf Subcase 1.1.} $\tilde{\hh}$ is a star.

    By symmetry we may assume $\tilde{\hh}$ is a star of center $2$. Then $|\tilde{\hh}|\leq \deg(2)\leq 4n$.

    \begin{claim}
    If there are $E_1,E_2,\ldots,E_t\in \tilde{\hh}$ such that $E_i\cap E_j=\{2\}$ for all $1\leq i<j\leq t$, then
    \begin{align}\label{ineq-3.2}
    e(G) \leq \binom{n-1}{2} -(t+1)n+t^2+3t+4.
    \end{align}
    \end{claim}
    \begin{proof}
    Let $X=E_1\cup E_2\cup \ldots\cup E_t$.
    By Claim \ref{claim s=2}, for each $x\in [2,n]\setminus X$ there are at most t edges in $G$ between $x$ and $X$. Moreover for $i=t,t-1,\ldots,2$, each $x\in E_i\setminus \{2\}$ there are at most $i-1$ edges in $G$ between $x$ and $E_1\cup \ldots\cup E_{i-1}$. Then
    \begin{align*}
    e(G) &\leq \binom{n-1}{2} -(t+1)(n-2t-2)-2\big(t+(t-1)+\cdots+2\big) \\[3pt]
    &=\binom{n-1}{2} -(t+1)n+t^2+3t+4.
    \end{align*}
    \end{proof}

    Let $\tilde{\hh}(2):=\{E\setminus \{2\}\colon E\in \tilde{\hh}\}$. If $\nu(\tilde{\hh}(2))\geq 4$, then one can find $E_1,E_2,E_3,E_4$ such that
such that $E_i\cap E_j=\{2\}$ for all $1\leq i<j\leq 4$. By \eqref{ineq-3.2} and $n\geq 60$,
     \[
     |\hh| = e(G)+|\tilde{\hh}| \leq  \binom{n-1}{2}-5n+32+4n <\binom{n-1}{2}.
     \]
    Thus we may assume $\nu(\tilde{\hh}(2))=t\in \{1,2,3\}$. Then by Theorem \ref{thm-eg} we have $|\tilde{\hh}|=|\tilde{\hh}(2)|\leq t(n-2)$.
     By \eqref{ineq-3.1},
     \begin{align*}
     |\hh| = e(G)+|\tilde{\hh}| &\leq   \binom{n-1}{2}-(t+1)n+t^2+3t+4+t(n-2)\\[2pt]
     &\leq \binom{n-1}{2}-n+16\\[2pt]
     &<\binom{n-1}{2}.
     \end{align*}

    {\bf Subcase 1.2.} $\tilde{\hh}$ is not a star.

By Theorem \ref{thm-HM} we have
\begin{align}\label{ineq-3.3}
|\tilde{\hh}|\leq \binom{n-2}{2}-\binom{n-3-2}{2}+1=3n-11.
\end{align}

Let
\[
\Delta_2(\tilde{\hh}) =\max_{\{u,v\}\subset [2,n]} |\{E\in\tilde{\hh}\colon \{u,v\}\subset E \}|.
\]
\begin{claim}\label{claim-sscase}
 We may assume that     $\Delta_2(\tilde{\hh}) < n/3$.
\end{claim}
\begin{proof}
    Let $\{u,v\} \subseteq [2,n]$ such that $|\{E\in\tilde{\hh}\colon \{u,v\}\subset E \}|$ is maximal and let $X=\{w\in [2,n]\colon uvw\in \tilde{\hh}\}$. Since $\hh$ is $F_5$-free,  $X$ has to be an independent set of $G$. If $|X| \geq n/3$, then for $n\geq 60$,
\begin{align*}
e(G) &\leq \binom{n-1}{2} - \binom{|X|}{2} \leq \binom{n-1}{2}- \frac{(n/3)(n/3 -1)}{2} < \binom{n-1}{2} -(3n-11).
\end{align*}
By \eqref{ineq-3.3},
\[
|\hh| = e(G)+|\tilde{\hh}|< \binom{n-1}{2} -(3n-11)+(3n-11)=\binom{n-1}{2}.
\]
Thus we may assume $|X|<n/3$.
  \end{proof}

Since  $\tilde{\hh}$ is not 2-intersecting, there exist $x_1x_2x_3,x_3x_4x_5\in \tilde{\hh}$. Since $\tilde{\hh}$ is not a star, there exists $y_1y_2y_3\in \tilde{\hh}$ with $x_3\notin \{y_1,y_2,y_3\}$. Then each edge of $\tilde{\hh}$ contains one of the following 7 pairs:
\[
\{x_1,x_4\},\{x_1,x_5\},\{x_2,x_4\},\{x_2,x_5\}, \{x_3,y_1\}, \{x_3,y_2\}, \{x_3,y_3\}.
\]
By Claim \ref{claim-sscase},
\[
|\tilde{\hh}|\leq 7\cdot\Delta_2(\tilde{\hh}) \leq \frac{7n}{3}.
\]
Using \eqref{ineq-3.1},
\[
|\hh| = e(G)+|\tilde{\hh}| \leq \binom{n-1}{2}-3n +14+\frac{7n}{3} <\binom{n-1}{2}.
\]

{\bf Case 2. } $i=0$.

Assume $S=A=\{1,2\}$. Then
\[
|\hh|\leq  \deg(1,[3,n])+\deg(2,[3,n])+ \deg(1,2).
\]
Since $\hh$ is $F_5$-free, we infer that $L(1,[3,n]),L(2,[3,n])$ are 2-colored triangle-free. By Theorem \ref{thm-2-colored},
\[
\deg(1,[3,n])+\deg(2,[3,n])\leq  2\left\lfloor\frac{(n-2)^2}{4}\right\rfloor \leq \frac{(n-2)^2}{2}.
\]
If $\deg(1,2) <\frac{n-2}{2}$, then
\[
|\hh|< \frac{(n-2)^2}{2}+\frac{n-2}{2}= \binom{n-1}{2}.
\]
Thus we may assume $\deg(1,2) \geq \frac{n-2}{2}\geq 3$.

Now by the $F_5$-free property, we infer that $L(1,[3,n])\cap L(2,[3,n])=\emptyset$. It follows that
\begin{align}\label{ineq-3.4}
|\hh|\leq  \deg(1,[3,n])+\deg(2,[3,n])+ \deg(1,2)\leq \binom{n-2}{2}+(n-2)=\binom{n-1}{2},
\end{align}
with equality holding if and only if $\deg(1,2)=n-2$ and $L(1,[3,n])\cap L(2,[3,n])=\binom{[3,n]}{2}$. Since $\nu(\hh)=2$ implies $L(1,[3,n])\neq \emptyset\neq  L(2,[3,n])$, there exist $xy\in L(1,[3,n])$ and $xz\in  L(2,[3,n])$. But then $1yx,1y2,2xz$ form an $F_5$, a contradiction. Thus we cannot have equality in \eqref{ineq-3.4} and equality holds if and only if $\hh$ is the full star.
\end{proof}

\section{The Proof of Theorem \ref{main1}}

\begin{proof}[Proof of Theorem \ref{main1}]
Let $\mathcal{H}\subset \binom{[n]}{3}$ be an $F_5$-free $ 3 $-graph with  $\nu(\hh)\leq s$. By Proposition \ref{prop-1} we may assume $s\geq 3$.
By Lemma \ref{lem-e(H)}, there exist $0\leq i \leq s$ and $A\subset S\subset [n]$ with $|A|=s-i$ and $|S|=s$ such that $\deg(x)\leq (s+2i)n$ for all $x\in \bar{A}$ and
\begin{align}\label{ineq-4.1}
|\hh|\leq \sum_{v\in A} \deg(v,\bar{S})+ \sum_{\{u,v\}\in \binom{A}{2}}\deg(u,v)+ is n+6i^2n.
\end{align}
Without loss of generality, we assume that $A=\{1,2,\ldots,s-i\}$ and $S=\{1,2,\ldots,s\}$.

If $i=s$, then $A=\emptyset$ and for  $n\geq 30s$ it follows from \eqref{ineq-4.1} that
 \[
  |\hh|\leq 7s^2n<s\cdot \frac{(n-s)^2}{4}-s\leq s\left\lfloor\frac{(n-s)^2}{4}\right\rfloor.
 \]

Now we distinguish three cases.

\vspace{2pt}
{\bf Case 1. } $i=s-1$.
\vspace{2pt}

 If $\hh[\bar{S}]\neq\emptyset$, then by Claim \ref{claim s=2} we have
\[
\deg(1,\bar{S})\leq \binom{n-s}{2}-2(n-s-3).
\]
 By \eqref{ineq-4.1} and $n\geq 30(s+1)$, we have
 \begin{align*}
   |\hh|&\leq \binom{n-s}{2}-2(n-s-3)+(s-1)sn+6(s-1)^2n\\[3pt]
   &<s\cdot\left(\frac{(n-s)^2}{4}-1\right) -\left(\frac{(s-2)(n-s)^2}{4}-s+2(n-s-3)-(s-1)sn-6(s-1)^2n\right).
 \end{align*}
 Note that
 \begin{align*}
 &\frac{(s-2)(n-s)^2}{4}-s+2(n-s-3)-(s-1)sn-6(s-1)^2n\\[3pt]
 = \ & \frac{1}{4}\left((s-2)n^2-(30 s^2- 56 s +16)n+s^3- 2 s^2-12 s-24\right)\\[3pt]
\geq\  & \frac{1}{4}\left(30(s-2)(s+1)n-(30 s^2- 56 s +16)n+s^3- 2 s^2-12 s-24\right)\\[3pt]
 =\  &\frac{1}{4}\left((26 s -76)n+s^3- 2 s^2-12 s-24\right)\\[3pt]
 \geq\  &\frac{1}{4}\left(30\times(26 s -76)(s+1)+s^3- 2 s^2-12 s-24\right)\\[3pt]
 \geq \ & \frac{1}{4}(s^3+ 778 s^2- 1512 s-2304)\\[3pt]
 \geq \  &\frac{1}{4}(s^3+ 778\times 3 s- 1512 s-2304)\\[3pt]
 \geq\  &\frac{1}{4}(s^3+ 822 s-2304)>0.
 \end{align*}
 Thus,
 \begin{align*}
  |\hh| <s\cdot\left(\frac{(n-s)^2}{4}-1\right) \leq s\left\lfloor\frac{(n-s)^2}{4}\right\rfloor.
 \end{align*}

 If $\hh[\bar{S}]=\emptyset$, noting that $\deg(v)\leq (3s-2)n$ for all $v\in \bar{A}$, then for $n\geq 30s$ we have
 \begin{align*}
   |\hh|\leq \deg(1,\bar{S})+\sum_{j=2}^{s}\deg(j)
   &\leq \binom{n-s}{2}+(s-1)(3s-2)n\\[2pt]
   &<s\cdot\frac{(n-s)^2}{4}-s \\[2pt]
   &\leq s\left\lfloor\frac{(n-s)^2}{4}\right\rfloor.
 \end{align*}

\vspace{2pt}
{\bf Case 2. } $\frac{s}{2}\leq i\leq s-2$.
\vspace{2pt}

Since $\hh$ is $F_5$-free, we infer that $L(1,\bar{S}),L(2,\bar{S}),\ldots,L(s-i,\bar{S})$ are 2-colored triangle-free. By Theorem \ref{thm-2-colored},
\[
 \sum_{v\in A} \deg(v,\bar{S}) \leq (s-i)\left\lfloor\frac{(n-s)^2}{4}\right\rfloor.
\]
Using \eqref{ineq-4.1}, we obtain that
\begin{align*}
	|\hh|&\leq (s-i)\left\lfloor\frac{(n-s)^2}{4}\right\rfloor+\binom{s-i}{2}n+6i^2n+isn\\[3pt]
          &= s\left\lfloor\frac{(n-s)^2}{4}\right\rfloor-i\left\lfloor\frac{(n-s)^2}{4}\right\rfloor+\frac{1}{2}(s-i)(s-i-1)n+6i^2n+isn\\[3pt]
          &= s\left\lfloor\frac{(n-s)^2}{4}\right\rfloor+\frac{1}{2}s(s-1)n+i\left(\frac{13}{2}in+\frac{n}{2}-\left\lfloor\frac{(n-s)^2}{4}\right\rfloor\right)\\[3pt]
          &\leq  s\left\lfloor\frac{(n-s)^2}{4}\right\rfloor+\frac{s(s-1)+13i^2+i}{2}n-\frac{i(n-s)^2}{4}+i.
	\end{align*}
Let $f(n,s,i)=\frac{s(s-1)+13i^2+i}{2}n-\frac{i(n-s)^2}{4}+i$.
Note that $f(n,s,i)$ is a convex function of $i$. By $n\geq 30s$ we have
\[
f(n,s,s-2)=-\frac{s-2}{4} n^2+\left(\frac{15}{2}s^2- 27 s +25\right)n-\left(\frac{s^3}{4}-\frac{s^2}{2}-s+2\right)<0
\]
and
\[
f(n,s,\frac{s}{2})=-\frac{s}{8} \left( n^2 -( 19 s-2) n + s^2-4\right)<0.
\]
Thus,
\begin{align*}
	|\hh| \leq  s\left\lfloor\frac{(n-s)^2}{4}\right\rfloor+\max\left\{f(n,s,s-2),f(n,s,\frac{s}{2})\right\}< s\left\lfloor\frac{(n-s)^2}{4}\right\rfloor.
\end{align*}

\vspace{2pt}
{\bf Case 3. } $ 0\leq i<\frac{s}{2}$.
\vspace{2pt}

Define
\[
T=\left\{v\in A\colon \deg(v,\bar{S})>\frac{(n-s)^2}{5}\right\}.
\]
Recall that $A=\{1,2,\ldots,s-i\}$ and $S=\{1,2,\ldots,s\}$.
By symmetry we may assume $T=\{1,2,\ldots,t\}$. Then for $j\in \{t+1,\ldots,s-i\}$ we have
\begin{align}\label{ineq-case2-1}
 \deg(j,\bar{S}) \leq \frac{(n-s)^2}{5}.
\end{align}

\vspace{2pt}
{\bf Subcase 3.1.} $|T|\leq \frac{s}{2}$.
\vspace{2pt}

Note that $s\geq 3$ and $i<\frac{s}{2}$ implies $s-i\geq 2$. Since $L(1,\bar{S}),L(2,\bar{S})$ are 2-colored triangle-free, by Theorem \ref{thm-2-colored} we have
\begin{align}\label{ineq-1}
 \deg(1,\bar{S})+\deg(2,\bar{S})\leq 2\left\lfloor\frac{(n-s)^2}{4}\right\rfloor.
\end{align}
If $|T|\leq 1$, then  by \eqref{ineq-4.1},  \eqref{ineq-case2-1} and \eqref{ineq-1} we have
\begin{align*}
	|\hh|\leq 2\left\lfloor\frac{(n-s)^{2}}{4}\right\rfloor+(s-i-2)\frac{(n-s)^2}{5}+\binom{s-i}{2}n+6i^2n+isn.
\end{align*}
Let $g(n,s,i)=-i\frac{(n-s)^2}{5}+\binom{s-i}{2}n+6i^2n+isn$. Then
\[
|\hh|\leq s\left\lfloor\frac{(n-s)^{2}}{4}\right\rfloor+s-\frac{1}{20}(s-2)(n-s)^2+g(n,s,i).
\]
For $n\geq 23s$, one can check that
\[
g(n,s,i)\leq \max\left\{g(n,s,0),g(n,s,\frac{s}{2})\right\}<\frac{1}{20}(s-2)(n-s)^2-s.
\]
Thus $|\hh|< s\left\lfloor\frac{(n-s)^{2}}{4}\right\rfloor$. Therefore we may assume $2\leq |T|\leq s-i$.

Next assume $2\leq |T|\leq \frac{s}{2}$. Since $L(1,\bar{S}),L(2,\bar{S}),\ldots,L(t,\bar{S})$ are 2-colored triangle-free, by Theorem \ref{thm-2-colored} we infer that
\begin{align}\label{ineq-case2-2}
 \sum_{v\in T} \deg(v,\bar{S}) \leq t\left\lfloor\frac{(n-s)^2}{4}\right\rfloor.
\end{align}
Using \eqref{ineq-4.1} and \eqref{ineq-case2-1}, we get
\begin{align*}
|\hh|&\leq t\left\lfloor\frac{(n-s)^{2}}{4}\right\rfloor+(s-i-t)\frac{(n-s)^2}{5}+\binom{s-i}{2}n+6i^2n+isn\\[3pt]
&\leq t\left\lfloor\frac{(n-s)^{2}}{4}\right\rfloor+(s-t)\frac{(n-s)^2}{5}+g(n,s,i)\\[3pt]
&\leq s\left\lfloor\frac{(n-s)^{2}}{4}\right\rfloor+s-(s-t)\frac{(n-s)^2}{20}+g(n,s,i)\\[3pt]
&\leq s\left\lfloor\frac{(n-s)^{2}}{4}\right\rfloor+s-\frac{s(n-s)^2}{40}+g(n,s,i).	
\end{align*}
By $n\geq 22s$, one can check that
\[
g(n,s,i)\leq \max\left\{ g(n,s,0),g(n,s,\frac{s}{2})\right\}<\frac{s(n-s)^2}{40}-s.
\]
Thus $|\hh|< s\left\lfloor\frac{(n-s)^{2}}{4}\right\rfloor$.

\vspace{2pt}
{\bf Subcase 3.2.} $|T|>\frac{s}{2}$ and there exist $x,y\in T$ such that $L(x,\bar{S})\cap L(y,\bar{S})= \emptyset$.
\vspace{2pt}

Since $\hh$ is $F_5$-free, we infer that $L(1,\bar{S}),L(2,\bar{S}),\ldots,L(s-i,\bar{S})$ are 2-colored triangle-free. Moreover,
\[
L(x,\bar{S})\cap L(y,\bar{S})\neq \emptyset,\ \deg(x,\bar{S})>\frac{(n-s)^2}{5}\mbox{ and } \deg(y,\bar{S})>\frac{(n-s)^2}{5}.
\]
Applying Theorem \ref{thm-disjoint} with $\alpha=\frac{1}{20} $ and $\beta=\frac{1}{28}$, we obtain that
\[
\sum_{v\in A} \deg(v,\bar{S}) \leq (s-i) \frac{(n-s)^2}{4} -\frac{1}{28}(s-i-2)  (n-s)^2.
\]
Then by  \eqref{ineq-4.1} we  obtain that
	\begin{align*}
	|\hh|&\leq (s-i) \frac{(n-s)^2}{4} -\frac{1}{28}(s-i-2)(n-s)^2+\binom{s-i}{2}n+6i^2n+isn.
\end{align*}
Let $h(n,s,i) = -i \frac{(n-s)^2}{4}+\frac{i}{28}(n-s)^2+\binom{s-i}{2}n+6i^2n+isn$. Then
\[
|\hh| \leq s \frac{(n-s)^2}{4} -\frac{1}{28}(s-2)(n-s)^2+h(n,s,i)\leq  s\left\lfloor\frac{(n-s)^2}{4}\right\rfloor+s-\frac{1}{28}(s-2)(n-s)^2+h(n,s,i).
\]
One can check that for $n\geq 30(s+1)$,
\[
h(n,s,i)\leq \max\left\{ h(n,s,0),h(n,s,\frac{s}{2})\right\}<\frac{1}{28}(s-2)(n-s)^2-s.
\]
Thus $|\hh|< s\left\lfloor\frac{(n-s)^2}{4}\right\rfloor$.

\vspace{2pt}
{\bf Subcase 3.3.}  $|T|>\frac{s}{2}$ and $L(x,\bar{S})\cap L(y,\bar{S})\neq \emptyset$ for any $x,y\in T$.
\vspace{2pt}

Assume $uv\in L(x,\bar{S})\cap L(y,\bar{S})$.  If $xyz \in \hh$, then by the $F_5$-free property we infer that $z\in \{u,v\}$. It follows that  $\deg(x,y)\leq 2$ for any $x,y\in T$. Thus,
\begin{align}\label{ineq-s2}
  \sum_{\{u,v\}\in \binom{A}{2}}\deg(u,v) \leq  \binom{s-i}{2}n -\binom{t}{2}(n-2).
\end{align}

\begin{claim}\label{claim3}
We may assume $|A|=s$.
\end{claim}
	 \begin{proof}
	 	Suppose for contradiction that $|A|\leq s-1$. Then by \eqref{ineq-4.1}, \eqref{ineq-case2-2} and \eqref{ineq-s2},	
\begin{align*}
	|\hh|&\leq t\left\lfloor\frac{(n-s)^2}{4}\right\rfloor+(s-i-t)\frac{(n-s)^2}{5}+ \binom{s-i}{2}n -\binom{t}{2}(n-2)+6i^2n+isn =:f(t).
\end{align*}
By $t\leq s-i$ and $n\geq 30(s+1)$,
\[
f'(t)= \left\lfloor\frac{(n-s)^2}{4}\right\rfloor-\frac{(n-s)^2}{5}-\frac{2t-1}{2}(n-2)\geq \frac{(n-s)^2}{20}-1-\frac{(2s-1)(n-2)}{2}\geq 0.
\]
Thus,
\begin{align*}
|\hh|\leq f(s-i)&= (s-i)\left\lfloor\frac{(n-s)^2}{4}\right\rfloor+2\binom{s-i}{2}+6i^2n+isn\\[3pt]
&\leq s\left\lfloor\frac{(n-s)^2}{4}\right\rfloor-i\frac{(n-s)^2}{4}+i+2\binom{s-i}{2}+6i^2n+isn.
\end{align*}
Let $q(n,s,i)=-i\frac{(n-s)^2}{4}+i+2\binom{s-i}{2}+6i^2n+isn$. Note that $|A|\leq s-1$ implies $i\geq 1$. Then for $n\geq 18s$ and $s\geq 3$,
\[
q(n,s,i) \leq \max\left\{q(n,s,1),q(n,s,\frac{s}{2})\right\}<0.
\]
Thus $|\hh|< s\lfloor\frac{(n-s)^{2}}{4}\rfloor$.
\end{proof}		
	
Now $S=A=\{1,2,\dots,s\}$. Then $i=0$ and by \eqref{ineq-4.1} we have
\begin{align}\label{ineq-4.1-2}
|\hh|\leq \sum_{v\in [s]} \deg(v,[s+1,n])+ \sum_{\{u,v\}\in \binom{[s]}{2}}\deg(u,v).
\end{align}

\begin{claim}\label{claim4}
We may assume $T=S$.
\end{claim}
	 \begin{proof}
If  $t\leq s-1$, then \eqref{ineq-4.1-2} and \eqref{ineq-s2},
\begin{align*}
|\hh|&\leq t\left\lfloor\frac{(n-s)^2}{4}\right\rfloor+(s-t)\frac{(n-s)^2}{5}+ \binom{s}{2}n -\binom{t}{2}(n-2)=:f(t).
\end{align*}
By $t\leq s-1$ and $n\geq 30(s+1)$,
\[
f'(t)= \left\lfloor\frac{(n-s)^2}{4}\right\rfloor-\frac{(n-s)^2}{5}-\frac{2t-1}{2}(n-2)\geq 0.
\]
Thus,
\begin{align*}
|\hh|\leq f(s-1)&\leq (s-1)\left\lfloor\frac{(n-s)^2}{4}\right\rfloor+\frac{(n-s)^2}{5}+2\binom{s-1}{2}+sn\\[2pt]
&\leq s\left\lfloor\frac{(n-s)^2}{4}\right\rfloor-\frac{(n-s)^2}{20}+1+2\binom{s-1}{2}+sn\\[2pt]
&< s\left\lfloor\frac{(n-s)^2}{4}\right\rfloor
\end{align*}	
and we are done.
\end{proof}		
	
Recall that  $T=\{v\in A\colon \deg(v,\bar{S})>\frac{(n-s)^2}{5}\}=A=S=\{1,2,\ldots,s\}$.

\begin{claim}\label{claim5}
We may assume that $T$ is an independent set.
\end{claim}
 \begin{proof}
Suppose $12v\in \hh$. Then by the $F_5$-free property, for any $x_1x_2\in L(1,\bar{S})\cap L(2,\bar{S})$ we have $\{v\}\cap \{x_1,x_2\}\neq \emptyset$. It follows that $|L(1,\bar{S})\cap L(2,\bar{S})|\leq n-s$.
Define $G_3$ as the square graph of $L(3,\bar{S})$ on the vertex set $\bar{S}$ such that  $xy\in E(G_3)$ if and only if there exists $z\in \bar{S}$ such that $xz,yz\in L(3,\bar{S})$. By the $F_5$-free property, we infer that $E(G_3)\cap (L(1,\bar{S})\cup L(2,\bar{S}))=\emptyset$.
 Since
 $|L(3,\bar{S})|>\frac{(n-s)^2}{5}$, by \eqref{ineq-degreepower} we obtain that
\begin{align*}
		e(G_3)>\frac{1}{n-s}\cdot\sum_{v\in \bar{S}}\binom{\deg_{L(3,\bar{S})}(v)}{2}
&=\frac{1}{2(n-s)} \sum_{v\in \bar{S}}\left(\deg_{L(3,\bar{S})}(v)^2-\deg_{L(3,\bar{S})}(v)\right)\\[3pt]
        &\geq \frac{1}{2(n-s)}\left(\frac{4|L(3,\bar{S})|^2}{n-s}-2|L(3,\bar{S})|\right)\\[3pt]
        &>\frac{2}{25}(n-s)^2-\frac{n-s}{5}.
	\end{align*}
Consequently,
\[
|L(1,\bar{S})\cup L(2,\bar{S})|\leq \binom{n-s}{2}-e(G_3)\leq \binom{n-s}{2}-\frac{2}{25}(n-s)^2+\frac{n-s}{5}
\]
and then
\begin{align*}
|L(1,\bar{S})|+ |L(2,\bar{S})|&\leq |L(1,\bar{S})\cup L(2,\bar{S})|+|L(1,\bar{S})\cap L(2,\bar{S})|\\[3pt]
&\leq \binom{n-s}{2}-\frac{2}{25}(n-s)^2+\frac{n-s}{5}+n-s\\[3pt]
&\leq
\frac{(n-s)^2}{2}-\frac{2}{25}(n-s)^2+n-s.
\end{align*}
By \eqref{ineq-4.1-2} and $n\geq 7s$, it follows that
\begin{align*}
	 	|\hh|
	 	&< (s-2)\left\lfloor\frac{(n-s)^2}{4}\right\rfloor+2\binom{s}{2}+\frac{(n-s)^2}{2}-\frac{2}{25}(n-s)^2+n-s\\[3pt]
        &<s\left\lfloor\frac{(n-s)^2}{4}\right\rfloor+n+s^2-2s-\frac{2}{25}(n-s)^2\\[3pt]
        &<s\left\lfloor\frac{(n-s)^2}{4}\right\rfloor.
	 	\end{align*}

\end{proof}	

Finally, by \eqref{ineq-4.1-2}  and Theorem \ref{thm-2-colored} we conclude that
\[
|\hh| = \sum_{v\in [s]} \deg(v,[s+1,n])\leq s\left\lfloor\frac{(n-s)^{2}}{4}\right\rfloor
\]
with equality holding if and only if $\hh=\mathcal{H}_{3}(n,s)$ up to isomorphism.
\end{proof}

\section{Concluding remarks}

In this paper, we determine the Tur\'{a}n number of $\{F_{5},M_{s+1}^{3}\}$ for $n\geq 30(s+1)$. It is an interesting question that whether this result holds for the full range of $n$.

\begin{conj}
For $n\geq 3(s+1)$ and $s\geq 3$,
\[
\ex_3(n,\{F_{5},M_{s+1}^{3}\}) = s \left\lfloor\frac{(n-s)^{2}}{4}\right\rfloor.
\]
\end{conj}

As a hypergraph analog of Theorem \ref{thm-2-colored}, we end up this paper with the following conjecture.

\begin{conj}
	Let $p\geq 2$ and let $\hh_{1},\ldots,\hh_{p}\subset\binom{[n]}{3}$ be 2-colored $F_5$-free. Then
	\[
	\sum_{i=1}^{p} |\hh_i|\leq \max\left\{\binom{n}{3},\ p \left\lfloor \frac{n}{3}\right\rfloor \left\lfloor \frac{n+1}{3}\right\rfloor \left\lfloor \frac{n+2}{3}\right\rfloor\right\}.
	\]
\end{conj}


\begin{thebibliography}{15}

\bibitem{AF}
J. Akiyama and P. Frankl, On the size of graphs with complete-factors, J. Graph Theory,  9 (1985), 197--201.

\bibitem{AlonFrankl}
N. Alon and P. Frankl, Tur\'{a}n graphs with bounded matching number, J. Combin. Theory Ser. B, 165 (2024), 223--229.


\bibitem{EG}
P. Erd\H{o}s and T. Gallai, On maximal paths and circuits of graphs, Acta. Math. Hung, 10 (1959), 337--356.


\bibitem{EKR}
P. Erd\H{o}s, C. Ko, and R. Rado, Intersection theorems for systems of finite sets, Quart. J. Math. Oxford Ser. 12 (1961), 313--320.

\bibitem{rainbow triangle}
P. Frankl, Graphs without rainbow triangles, arXiv:2203.07768v1, (2022).

\bibitem{F-EKR}
P. Frankl, The Erd\H{o}s-Ko-Rado theorem is true for $n = ckt$, Coll. Math. Soc. J. Bolyai, 18 (1978), 365--375.
\bibitem{rainbowtriangle2}
P. Frankl, E. Gy\H{o}ri, Z. He, Z. Lv, N. Salia, C. Tompkins, K. Varga, X. Zhu, Some remarks on
graphs without rainbow triangles, arXiv:2204.07567, (2022).

\bibitem{FranklFuredi}
P. Frankl and Z. F\"{u}redi, A new generalization of the Erd\H{o}s--Ko--Rado theorem, Combinatorica, 3 (1983), 341--349.

\bibitem{FWY}
L. Fu, J. Wang, and  W. Yang, The maximum number of edges in a $\{K_{r+1},M_{k+1}\}$-free graph, Discuss. Math. Graph Theory, 44 (2024), 1617--1629.

\bibitem{GTZ}
D. Gerbner, C. Tompkins, and J. Zhou, On hypergraph Tur\'an problems with bounded matching number, Eur. J. Combin. 127 (2024), 104155.


\bibitem{HM}
A.J.W. Hilton and E.C. Milner, Some intersection theorems for systems of finite sets, Quart.J.Math. Oxford Ser, 18 (1967), 369--384.

\bibitem{LMZ}
X. Li, J. Ma, and Z. Zheng, On the multicolor Tur\'{a}n conjecture for color-critical graphs, arXiv:2407.14905, (2024).


\bibitem{PeterMubayi}
P. Keevash and D. Mubayi,  Stability theorems for cancellative hypergraphs, J. Combin. Theory Ser. B, 92 (2004), 163--175.

\bibitem{KSSV}
P. Keevash, M. Saks, B. Sudakov and J. Verstraete, Multicolour Tur\'{a}n
 problems, Advances in Applied Mathematics, 33 (2004), 238--262.
 \bibitem{MH}
 Y. Ma and X. M. Hou, Graphs without rainbow cliques of orders four and five, arXiv:2306.12222v1, (2023).
\bibitem{turan1941}
P. Tur\'{a}n, Eine Extremalaufgabe aus der Graphentheorie, Mat. Fiz. Lapok, 48 (1941), 436--452.

\bibitem{W-EKR}
R. M. Wilson, The exact bound in the Erd\H{o}s-Ko-Rado theorem, Combinatorica 4 (1984), 247--257.


\end{thebibliography}
\end{document}